\documentclass[12pt, a4paper, reqno]{amsart}

\usepackage{amsmath, amssymb, amsthm}
\usepackage[T1]{fontenc}
\usepackage[utf8]{inputenc}
\usepackage[all]{xy}
\usepackage[
  margin=1.5cm,
  includefoot,
  footskip=30pt,
  ]{geometry}
\usepackage[colorlinks=true,citecolor=cyan,linkcolor=magenta, urlcolor=blue, filecolor=green, backref=page]{hyperref}
\usepackage{hyperref}
\usepackage{indentfirst}
\usepackage[last]{changes}
\definechangesauthor[color=blue]{KY}

\newcommand{\cL}{\mathcal{L}}
\newcommand{\cO}{\mathcal{O}}

\newcommand{\kbar}{\overline{k}}

\DeclareMathOperator{\Aut}{Aut}

\DeclareMathOperator{\End}{End}
\DeclareMathOperator{\Frob}{Frob}
\DeclareMathOperator{\Gal}{Gal}

\DeclareMathOperator{\Hom}{Hom}

\DeclareMathOperator{\Spec}{Spec}
\DeclareMathOperator{\Twist}{Twist}

\DeclareMathOperator{\ord}{ord}

\theoremstyle{plain}
\newtheorem{theorem}{Theorem}[section]
\newtheorem{lemma}[theorem]{Lemma}
\newtheorem{proposition}[theorem]{Proposition}
\newtheorem{corollary}[theorem]{Corollary}
\newtheorem{definition}[theorem]{Definition}
\newtheorem{remark}[theorem]{Remark}

\theoremstyle{definition} 

\newtheorem{example}[theorem]{Example}

\theoremstyle{remark} 

\DeclareFontFamily{U}{wncy}{}
\DeclareFontShape{U}{wncy}{m}{n}{<->wncyr10}{}
\DeclareSymbolFont{mcy}{U}{wncy}{m}{n}
\DeclareMathSymbol{\Sha}{\mathord}{mcy}{"58}

\newcommand{\lp}{\left(}
\newcommand{\rp}{\right)}

\newcommand{\lbrb}[1]{\lp #1 \rp}
\newcommand{\lcrc}[1]{\{ #1 \}}

\title{Counting the number of the twists of certain polarized abelian varieties}
\author{Wontae Hwang and Keunyoung Jeong}
\address{School of Mathematics, Korea Institute for Advanced Study, 85 Hoegiro, Dongdaemun-gu, Seoul 02455, South Korea}
\email{hwangwon@kias.re.kr}
\address{Department of Mathematical Sciences, Ulsan National Institute of Science and Technology, UNIST-gil 50, Ulsan 44919, South Korea}
\email{kyjeongg@gmail.com}

\begin{document}



\maketitle

\begin{abstract}
We count the number of isomorphism classes of degree $d$-twists of some polarized abelian varieties over finite fields of odd prime dimension. This can be seen as a higher dimensional analogue of the counting problem for elliptic curves case.
\end{abstract}

\section{Introduction}\label{intro}
Let $k$ be a perfect field and let $K$ be a Galois extension of $k.$ For a quasi-projective algebraic variety $X$ over $k$, an algebraic variety $X^{\prime}$ over $k$ which admits a $K$-isomorphism $\phi : X^{\prime}_K \rightarrow X_K$ is called a \emph{$(K/k)$-twist} of $X.$ 
We will denote the set of $k$-isomorphism classes of $(K/k)$-twists of $X$ by $\Twist(K/k, X)$. 
Then it is well known \cite[III. \S 1, Proposition 5]{Ser2} that there is a bijection between the set $\Twist(K/k, X)$ and the first Galois cohomology group $H^1(\Gal(K/k), \Aut_{K}(X_K))$.
In this paper, we consider the case when $X$ is an abelian variety over $k$.
As we mentioned above, there is a bijection between $\Twist(K/k, X)$ and $H^1(\Gal(K/k), \Aut_K(X_K))$. 
In general, $\Aut_K(X_K)$ is infinite, which makes it difficult to compute the group $H^1(\Gal(K/k), \Aut_K(X_K))$.
To remedy this difficulty, we introduce a polarization $\lambda$ on $X$ over $k$ to consider a polarized abelian variety $(X,\lambda)$ so that $\Aut_K(X_K, \lambda_K)$ is finite.
Then, analogously, we have the following bijection:
\begin{equation} \label{main eqn}
\theta \colon \Twist(K/k, (X, \lambda)) \to H^1(\Gal(K/k), \Aut_{K}(X_{K}, \lambda_{K})).
\end{equation}
The definitions of twists and automorphism groups of polarized abelian varieties are provided in Section \ref{twists of polar}.
Using (\ref{main eqn}), we provide an answer to the problem of counting the number of twists of certain polarized abelian varieties over finite fields, together with a work of the first author \cite{Hwa2}.  
In fact, we give a finer result by taking the degree of the twists into account.

To this aim, we 
introduce the notion of the degree of twists (see Definition \ref{degdef}): we say that a $(\overline{k}/k)$-twist $(X^{\prime},\lambda^{\prime})$ of $(X, \lambda)$ is of degree $n$ if the ``minimal'' defining field $L$ of the corresponding $\overline{k}$-isomorphism $\phi : (X^{\prime}_{\overline{k}}, \lambda^{\prime}_{\overline{k}}) \rightarrow (X_{\overline{k}}, \lambda_{\overline{k}})$ is of degree $n$ over $k.$
In the case of elliptic curves, the number of $(\overline{k}/k)$-twists of $(E,O_E)$ of degree $d$, up to $k$-isomorphism, can be described explicitly. (For a more precise statement, see Proposition \ref{ECtwists} below.) 
The main theorem of this paper is a higher dimensional analogue of the previously mentioned result.

\begin{theorem} \label{main theorem}
Let $(X,\lambda)$ be an absolutely simple polarized abelian variety of odd prime dimension $g$ over a finite field $k$ with characteristic $\geq 5$, and let $s$ be an integer such that an action of the Frobenius automorphism $\Frob_k$ on $\Aut_{\kbar}(X_{\kbar}, \lambda_{\kbar})$ is induced by a multiplication by $s$. 
Also, let $N(d)$ be the number of $(\kbar/k)$-twists of $(X, \lambda)$ of degree $d$, $N=\sum_{d\geq 1}N(d)$, and $n = |\Aut_{\kbar}(X_{\kbar}, \lambda_{\kbar})|$.
Then we have
\begin{equation*}
N(d) = 
\prod_{\substack{p \mid d }} \lbrb{
p^{\min(t_{1,p}, t_{3,p}) + \min(t_{2,p}, t_{3,p}) - \min(t_{1,p} + t_{2,p}, t_{3,p})} - p^{\min(t_{1,p}-1, t_{3,p}) + \min(t_{2,p}, t_{3,p}) - \min(t_{1,p} + t_{2,p} - 1, t_{3,p})}},
\end{equation*}
for $d \mid n$, where $t_{1,p}=\ord_p(s-1), t_{2,p}=\ord_{p}(s^{d-1}+\cdots+1),$ and $t_{3,p}=\ord_{p} n$,
and $N(d) = 0$ otherwise.
\end{theorem}
To prove the theorem, we combine the structure theory of $H^1(G, M)$ for pro-cyclic groups $G$ and cyclic $G$-modules $M$, and the M\"{o}bius inversion formula, together with the facts that the bijection $\theta$ of (\ref{main eqn}) fits into the inflation-restriction sequence and the group $\textrm{Aut}_{\overline{k}}(X_{\kbar}, \lambda_{\kbar})$ is cyclic in this case. \\

This paper is organized as follows:
it seems that (\ref{main eqn}) is known, and the only reference \cite[Theorem IV.7.10]{Mil} that we could have found gives a somewhat brief proof. Hence, in Section \ref{twists of polar}, we give a more detailed proof of (\ref{main eqn}) in Theorem \ref{Intromain1}.
In Section \ref{main result}, as mentioned in the above, 
we first introduce the notion of the degree of a $(\overline{k}/k)$-twist of a polarized abelian variety $(X,\lambda)$ over $k$, and state a result on the number of $k$-isomorphism classes of $(\overline{k}/k)$-twists of a given elliptic curve $(E,O_E)$ over $k$. Afterwards, we prove the cyclicity of the automorphism groups of certain absolutely simple polarized abelian varieties of odd prime dimension over finite fields, and provide the proof of Theorem \ref{main theorem}, together with a concrete example.

\indent In the sequel, let $k$ be a perfect field, $K$ a Galois extension of $k$, $\overline{k}$ an algebraic closure of $k$, and $G_k = \textrm{Gal}(\overline{k}/k),$ the absolute Galois group of $k,$ unless otherwise stated. Also, for an integer $n \geq 1,$ let $\varphi(n)$ denote the number of positive integers that are less than or equal to $n$ and relatively prime to $n,$ and $\mu_n$ the group scheme of $n$-th roots of unity over $k.$ Finally, for two positive integers $m$ and $n$, $(m,n)$ denotes the greatest common divisor of $m$ and $n$, and for an automorphism $f$ of a group $A$, we define $A[f]$ to be the kernel of $f$.

\section{Twists of polarized abelian varieties}\label{twists of polar}
In this section, we introduce the notion of twists of polarized abelian varieties over $k$. To this aim, we first recall the following fundamental result, which was mentioned in Section \ref{intro}.
\begin{proposition}\label{TwistH1}
Let $X$ be a quasi-projective algebraic variety over $k$ and let $\Twist(K/k, X)$ be the set of $k$-isomorphism classes of $(K/k)$-twists of $X$. Then there is a bijective map
$$\theta \colon \Twist(K/k, X) \to H^1(\Gal(K/k), \Aut_K(X_K)).$$
\end{proposition}
\begin{proof}
For a proof, see \cite[Chapter III. \S 1, Proposition 5]{Ser2}.
\end{proof}

Now, we consider a polarized abelian variety $(X, \lambda)$ over $k$: let $X$ be an abelian variety over $k$, and let $\lambda \colon X \to X^\vee$ be a polarization on $X$ that is defined over $k$, where $X^\vee$ denotes the dual abelian variety of $X$. 
We define a homomorphism between two polarized abelian varieties $(X, \lambda)$ and $(X^{\prime}, \lambda^{\prime})$ over $k$ as a homomorphism of abelian varieties $\phi \colon X \to X^{\prime}$ over $k$ that 
satisfies $\lambda = \phi^\vee \circ \lambda' \circ \phi$. A homomorphism $\phi \colon (X,\lambda) \rightarrow (X^{\prime},\lambda^{\prime})$ of polarized abelian varieties over $k$ is called \emph{a $k$-isomorphism} if $\phi$ is defined over $k$ and it admits the inverse homomorphism $\psi \colon (X^{\prime},\lambda^{\prime}) \rightarrow (X,\lambda)$, which is also defined over $k$.

Next, we define the notion of twists of polarized abelian varieties. Let $(X,\lambda)$ be a polarized abelian variety over $k$. Then we introduce the following
\begin{definition} \label{polartwistdef}
A pair $(X^{\prime}, \lambda')$ is called a \emph{$(K/k)$-twist of $(X,\lambda)$} if $(X', \lambda')$ is a polarized abelian variety over $k$, and there exists a $K$-isomorphism $\phi \colon (X'_K, \lambda'_K) \to (X_K, \lambda_K)$ of polarized abelian varieties. The set of $k$-isomorphism classes of $(K/k)$-twists of $(X,\lambda)$ is denoted by $\Twist\lbrb{K/k, (X, \lambda)}$.
\end{definition}

To obtain the main result of this section, we further recall that an element $\sigma \in \Gal(K/k)$ defines a morphism $\Spec \sigma \colon \Spec K \to \Spec K$, which will be written as $\sigma$ (for notational convenience), and induces a $K$-isomorphism $\sigma_{X_K} \colon X_K \to X_K$ via the universal property of the fiber product of $k$-schemes. 

Finally, we also recall the notion of the relative Picard scheme. Let $S$ be a scheme. Following \cite[\S 6.1]{Ben}, we consider a quasi-compact and quasi-separated $S$-scheme $f \colon Y \to S$ with a section $\epsilon \colon S \to Y$, satisfying $f_*(\mathcal{O}_{Y \times_S T}) = \mathcal{O}_T$ for all $S$-schemes $T$.
Let $\cL$ be a line bundle on $Y_T$.
An isomorphism $\alpha \colon \cO_T \to \epsilon_T^*\cL$ is called \emph{a rigidification of $\cL$ along $\epsilon_T$}, and $(\cL, \epsilon_T)$ is called \emph{a rigidified line bundle on $Y_T$}. For rigidified line bundles $(\cL_1, \alpha_1)$ and $(\cL_2, \alpha_2)$ on $Y_T$, a homomorphism between rigidified line bundles is defined as a homomorphism $h \colon \cL_1 \to \cL_2$ of line bundles that satisfies
\begin{equation*}
(\epsilon_T^*h)\circ \alpha_1 = \alpha_2.
\end{equation*}

Now, we consider a functor $P_{Y/S, \epsilon}$ which maps an $S$-scheme $T$ to the group of the isomorphism classes of rigidified line bundles on $Y \times_S T$.
For the case when $S=\Spec K$ and $f \colon Y \to S$ is proper, it is known that the functor $P_{Y/S, \epsilon}$ is representable by an $S$-scheme. (See \cite[(6.3)]{Ben}.) 
In this case, the relative Picard scheme of $f \colon Y \to S$ is isomorphic to the representing object of $P_{Y/S, \epsilon}$, 
and we will write $\widetilde{f}$ for a structure morphism of the representing $S$-scheme of the functor $P_{Y/S, \epsilon}$. 
Throughout this paper, we focus on the case when $S=\Spec K$, and either $Y=\Spec K$ or $Y$ is an abelian variety over $K$.

Then we observe the following useful fact.

\begin{lemma} \label{sigmavee}
For any $\sigma \in \Gal(K/k)$, we have $\sigma_{X_K}^\vee = (\sigma_{X_K^\vee})^{-1}$.
\end{lemma} 
\begin{proof}
First of all, we note that $(\sigma_{X_K^\vee})^{-1} = (\sigma^{-1})_{X_K^\vee}$.
The relative Picard scheme of $\sigma \colon \Spec K \to \Spec K$ whose section $\epsilon$ equals $\sigma^{-1}$,
is the representing object of a functor $P_{\Spec K/\Spec K, \sigma^{-1}}$ that will be written as $P_{\sigma, \sigma^{-1}}$. By the definition of $\widetilde{\sigma}$, two functors $\Hom_K(-, \widetilde{\sigma})$ and $P_{\sigma, \sigma^{-1}}(-)$ are naturally isomorphic.

We will show that $\widetilde{\sigma} = \sigma$. Indeed, let $g \colon T \to \Spec K$ be a $K$-scheme. For $\epsilon = \sigma^{-1}$ on $\Spec K$, 
we define a map $\epsilon_g \colon T \to T\times_{g, K, \sigma}\Spec K$ by the universal property for $T$, $g \colon T \to \Spec K$, and $1_T$. In particular, we have that $p_1 \circ \epsilon_g = 1_T$ and $p_1$ is an isomorphism of $K$-schemes where $p_1 \colon T \times_{g,K,\sigma} \textrm{Spec}~K \rightarrow T$ is the first projection.

Let $(\mathcal{L}_1, \alpha_1)$ and $(\mathcal{L}_2, \alpha_2)$ be two elements in $P_{\sigma,\sigma^{-1}}(g)$ for  a $K$-scheme $g \colon T \to \Spec K$.
Then there is an isomorphism
\begin{equation*}
\alpha_2 \circ \alpha_1^{-1} \colon \epsilon_g^* \mathcal{L}_1 \to \epsilon_g^* \mathcal{L}_2.
\end{equation*}
Now, $p_1^*(\alpha_2 \circ \alpha_1^{-1})$
is an isomorphism of line bundles from $\cL_1$ to $\cL_2$ and satisfies 
\begin{equation} \label{rigidisom}
\epsilon_g^*(p_1^*(\alpha_2\circ\alpha_1^{-1})))\circ \alpha_1 = \alpha_2.
\end{equation}
Hence, for every $K$-scheme $g \colon T \to \Spec K$, the group of rigidified line bundles $P_{\sigma, \sigma^{-1}}(g)$ is trivial.
We choose a representative $(\mathcal{O}_{g \times \sigma}, \alpha)$, where $\mathcal{O}_{g \times \sigma}$ is the structure sheaf of
$T \times_{g, K, \sigma} \Spec K$ and $\alpha \colon \mathcal{O}_T \to \epsilon_g^*\mathcal{O}_{g \times \sigma}$ is a rigidification.

Now, $p_2 \circ p_{1}^{-1}$ is an element in $\Hom_K(g, \sigma)$.
Since $\sigma$ is isomorphic to $1_K$, which is the terminal object of the category of $K$-schemes,
this is the unique object of $\Hom_K(g, \sigma)$.
We define an isomorphism $\eta_g \colon \Hom_K(g, \sigma) \to P_{\sigma, \sigma^{-1}}(g)$ by
$$ \eta_g(p_2 \circ p_1^{-1}) = (p_2^* \cO_{\Spec K}, \beta)$$ 
where $\beta \colon \cO_T \to \epsilon_g^*p_2^* \cO_{\Spec K}$ is a rigidification.
(Note that $(p_2^*\cO_{\Spec K},\beta)$ is isomorphic to $(\cO_{g\times \sigma}, \alpha)$.)

The claim is that this $\eta$ gives a natural isomorphism between the functors $\Hom_K(-,\sigma)$ and $P_{\sigma, \sigma^{-1}}(-)$. 
Indeed, let $h \colon T' \to \Spec K$ be another $K$-scheme and $\phi \colon g \to h$ be a $K$-homomorphism.
We also use $p_{1,g}, p_{2, g}, p_{1, h}$ and $p_{2,h}$ to distinguish the projection morphisms of schemes $T \times_{g, K, \sigma} \Spec K$ and $T' \times_{h, K, \sigma} \Spec K$.
Then we can lift $\phi$ to $T \times_{g,K,\sigma} \Spec K \to T' \times_{h,K,\sigma} \Spec K$ 
by the universal property so that we have $p_{2,h} \circ \phi = p_{2,g}$ and $\phi \circ p_{1,g} = p_{1,h} \circ \phi$.
Then $\phi^* \colon \Hom_K(h, \sigma) \to \Hom_K(g, \sigma)$ is defined by
$$\phi^*(p_{2,h} \circ p_{1,h}^{-1})
= p_{2,h} \circ p_{1,h}^{-1} \circ \phi,  $$
where the latter is equal to $p_{2,g} \circ p_{1,g}^{-1}$ because we have
$$p_{2,h} \circ p_{1,h}^{-1} \circ \phi 
= p_{2,h} \circ \phi \circ p_{1,g}^{-1}
= p_{2,g} \circ p_{1,g}^{-1}.$$
On the other hand, the corresponding map $\phi^* \colon P_{\sigma, \sigma^{-1}}(h) \to P_{\sigma, \sigma^{-1}}(g)$ is given by
$$ \phi^*(p_{2,h}^* \mathcal{O}_{\Spec K}, \gamma) = (\phi^*p_{2,h}^* \mathcal{O}_{\Spec K}, \phi^*\gamma ),$$
where $\gamma \colon \mathcal{O}_{T'} \rightarrow \epsilon_h^* p_{2,h}^* \mathcal{O}_{\Spec K}$ is a rigidification. 
Since we have $p_{2,h} \circ \phi = p_{2,g}$, the pullback $\phi^*\gamma \colon \cO_T \to \epsilon_g^*p_{2,g}^* \cO_{\Spec K}$ is a rigidification of $p_{2,g}^*\cO_{\Spec K}$. 
Furthermore, we can see that $\phi^*(p_{2,h}^*\cO_{\Spec K}, \gamma)$ is isomorphic to $(p_{2,g}^*\cO_{\Spec K}, \beta)$, as in (\ref{rigidisom}). Therefore we get that $\eta_g \circ \phi^* = \phi^* \circ \eta_h$.
Since $\eta$ gives an isomorphism for any $K$-scheme, two functors $\Hom_K(-, \sigma)$ and $P_{\sigma, \sigma^{-1}}(-)$ are naturally isomorphic, and hence, we have $\widetilde{\sigma} = \sigma$.\\
 \indent Now, let $\epsilon \colon \Spec K \to X_K$ be a section of $f \colon X_K \to \Spec K$.
By taking dual, we have $\epsilon^\vee = \widetilde{f}$.
Similarly, by taking dual on $\sigma = f \circ (\sigma_{X_K} \circ \epsilon)$, we have 
$\widetilde{\sigma} \circ (\sigma_{X_K} \circ \epsilon)^\vee = \widetilde{f}$, which, in turn, implies that $\sigma_{X_K}^\vee = (\sigma^{-1})_{X_K^\vee}$. \\
\indent This completes the proof.
\end{proof}

As a consequence of the above lemma, we obtain the following.
\begin{corollary}\label{sigmaveecor}
Let $\phi \colon (X_K^{\prime}, \lambda_K^{\prime}) \rightarrow (X_K , \lambda_K)$ be a $K$-isomorphism of polarized abelian varieties and let $\sigma \in \Gal(K/k).$ Let ${^{\sigma}\phi}= \sigma_{X_K} \circ \phi \circ \sigma^{-1}_{X_K'}$. Then the map ${^{\sigma}\phi} \colon (X_K^{\prime}, \lambda_K^{\prime}) \rightarrow (X_K , \lambda_K)$ is also a $K$-isomorphism of polarized abelian varieties. Furthermore, we have $({}^\sigma\phi)^\vee = {}^\sigma(\phi^\vee)$.
\end{corollary}
\begin{proof}
By construction, it suffices to prove that ${^{\sigma} \phi}$ preserves the polarizations. By Lemma \ref{sigmavee}, together with the fact that $\lambda$ is defined over $k$, we have
$$\sigma_{X_K}^\vee \circ \lambda_K \circ \sigma_{X_K} = \sigma_{X_K^\vee}^{-1} \circ \lambda_K \circ \sigma_{X_K}
= \lambda_K,$$
i.e.\ we get that $\sigma_{X_K}^* \lambda_K = \lambda_K$. By a similar argument, we also have $(\sigma_{X^{\prime}_K}^{-1})^* \lambda_{K}'=\lambda_{K}'.$ Hence, it follows that
$$({}^\sigma\phi)^*\lambda_K = (\sigma_{X_K}\circ\phi\circ\sigma_{X_K'}^{-1})^*\lambda_K = (\sigma_{X^{\prime}_K}^{-1})^* (\phi^* (\sigma_{X_K}^* \lambda_K))= (\sigma_{X^{\prime}_K}^{-1})^* (\phi^* \lambda_K)= (\sigma_{X^{\prime}_K}^{-1})^* \lambda_K'=\lambda_K'.$$
For the second assertion, note that we have
$$ ({}^\sigma \phi)^\vee =
 (\sigma_{X_K} \circ \phi \circ \sigma_{X_K'}^{-1})^\vee
= (\sigma_{X_K'}^{-1})^\vee \circ \phi^\vee \circ (\sigma_{X_K})^\vee 
= \sigma_{(X_K')^{\vee}} \circ \phi^\vee \circ \sigma_{X_K^\vee}^{-1} 
= {}^\sigma(\phi^\vee),$$
by Lemma \ref{sigmavee}. \\
\indent This completes the proof.
\end{proof}

Now, we are ready to prove the main result of this section.

\begin{theorem} \label{Intromain1}
Let $\Twist(K/k, (X, \lambda))$ be the set of $k$-isomorphism classes of $(K/k)$-twists of a polarized abelian variety $(X, \lambda)$ over $k$.
Then there is a bijective map
$$\theta \colon \Twist(K/k, (X, \lambda)) \to H^1(\Gal(K/k), \Aut_{K}(X_{K}, \lambda_{K})).$$
\end{theorem}
\begin{proof}
For an element $(X', \lambda')$ of $\mathrm{Twist}(K/k, (X, \lambda))$, we let $\theta(X^{\prime}, \lambda^{\prime})(\sigma):= \theta(\phi)(\sigma) = {^{\sigma}\phi} \circ \phi^{-1}$ for all $\sigma \in \Gal(K/k)$, where $\phi \colon (X_K^{\prime}, \lambda_K^{\prime}) \rightarrow (X_K, \lambda_K)$ is a $K$-isomorphism. 
With the help of Corollary \ref{sigmaveecor}, we can see that ${^\sigma}\phi$ is a $K$-homomorphism of polarized abelian varieties, and
\begin{equation*}
({^\sigma}\phi \circ \phi^{-1})^\vee \circ \lambda_K \circ ({^\sigma}\phi \circ \phi^{-1})
=(\phi^{-1})^\vee \circ ({^\sigma}\phi)^\vee \circ \lambda_K \circ {^\sigma}\phi \circ \phi^{-1}
= \lambda_K,
\end{equation*} 
which implies that $\theta(\phi)(\sigma) \in \Aut_K(X_K, \lambda_K)$ for all $\sigma$. Also, note that we have
\begin{equation*}
\theta(\phi)(\sigma\tau) = {^{\sigma \tau}\phi}\circ\phi^{-1} = {^\tau ({^{\sigma}\phi}\circ\phi^{-1})}
\circ({^{\tau}\phi}\circ\phi^{-1}) = {^\tau (\theta(\phi)(\sigma))} \theta(\phi)(\tau)
\end{equation*}
as $K$-automorphisms of the polarized abelian variety $(X_K, \lambda_K).$ Now, we can proceed in a similar fashion as in the proof of \cite[Chapter III. $\S1$, Proposition 5]{Ser2} to obtain the desired result. \\
\indent This completes the proof.
\end{proof}

As an one dimensional analogue of Theorem \ref{Intromain1}, we recall the following well-known result.
\begin{proposition}\label{ell rem}
Consider an elliptic curve $(E, O_E)$ over $k$ where $O_E$ is the identity of $E.$ Then there is a bijective map between $\Twist(\overline{k}/k, (E,O_E))$ and $H^1 (G_k, \Aut_{\overline{k}}(E_{\overline{k}})).$ 
\end{proposition}
For a detailed discussion, see \cite[\S X.5]{Sil1}.



 
\section{Main result} \label{main result}
In this section, we give the main result of this paper. Throughout this section, let $k$ be a finite field of characteristic $p>0$. We begin with introducing the following notion.
\begin{definition} \label{degdef}
Let $(X,\lambda)$ be a polarized abelian variety over $k$ and let $(X^{\prime}, \lambda^{\prime})$ be a $(\overline{k}/k)$-twist of $(X, \lambda)$. 
We define the degree of $(X^{\prime}, \lambda^{\prime})$ to be the quantity
$$\min\lcrc{ [L:k] ~|~ \textrm{there is an $L$-isomorphism}~ \phi \colon (X'_L, \lambda'_L) \to (X_L, \lambda_L) ~\textrm{of polarized abelian varieties} },$$ 
and denote it by $\mathrm{deg}(X^{\prime},\lambda^{\prime}).$
\end{definition}

For example, a $(\overline{k}/k)$-twist $(X',\lambda')$ of $(X,\lambda)$ is of degree one if and only if it is $k$-isomorphic to $(X,\lambda)$. \\

Now, for an integer $m \geq 1,$ let $M= |\mu_m (k)|$. Then we have $\mu_m(k) = \mu_M(k)$, and for all integers $d \mid M$, we get that $\mu_d(\overline{k}) \subset k^\times$.
Then since $G_k$ acts trivially on $\mu_d(\overline{k})$, a character in $\Hom(G_k, \mu_d(\overline{k}))$ actually defines a cocycle in $H^1(G_k, \mu_d(\overline{k}))$. We also note that 
$$ H^1(G_k, \mu_d(\overline{k})) \hookrightarrow H^1(G_k, \mu_M(\overline{k})) \cong H^1(G_k, \mu_m(\overline{k})),$$
as abelian groups. Indeed, since the sequence
$$\xymatrix{1 \ar[r] & \mu_d(k) \ar[r] &\mu_M(k) \ar[r] & (\mu_{M}/\mu_d)(k) \ar[r] & H^1(G_k, \mu_d(\overline{k})) \ar[r] & H^1(G_k, \mu_M(\overline{k}))} $$
is exact, and the map $\mu_M(k) \to (\mu_{M}/\mu_d)(k)$ is surjective, we get that $H^1(G_k, \mu_d(\overline{k})) \hookrightarrow H^1(G_k, \mu_M(\overline{k}))$.
On the other hand, since $\mu_M(\overline{k}) \subset k^\times$ and $\mu_M(k) = \mu_m(k)$, we have
$$ k^\times/(k^\times)^M = k^\times/(k^\times)^m,$$
and each one is isomorphic to $H^1(G_k, \mu_M)$ and $H^1(G_k, \mu_m)$, respectively, by Hilbert's 90.
More concretely, for each $j=m,M,$ an element $\alpha \in k^\times$ gives rise to cocycles
\begin{equation*}
c_{\alpha, m}(\sigma) = \sigma(\beta)/\beta \quad \textrm{and} \quad c_{\alpha, M}(\sigma) = \sigma(\gamma)/\gamma
\end{equation*}
for some $\beta, \gamma \in \kbar^\times$ such that $\beta^m = \alpha$ and $\gamma^M = \alpha$,
and the isomorphism between $H^1(G_k, \mu_m)$ and $H^1(G_k, \mu_M)$ maps $c_{\alpha, m}$ to $c_{\alpha, M}$. \\

Now, we consider the problem of counting the number of twists of elliptic curves over finite fields with a given degree. 
For example, it is well known  that the number of degree 2 twists of elliptic curves $E$ over finite fields such that $j(E) \neq 0, 1728$ is equal to one, and we can generalize this fact easily for the twists of elliptic curves over finite fields of arbitrary degree as in Proposition \ref{ECtwists} below.
To this aim, we recall that if $\textrm{char}~k=p \geq 5$, then $\Aut_{\overline{k}}(E_{\kbar}) \cong \mu_m (\overline{k})$ as $G_k$-modules, 
where $m = 2, 4, 6$ if $j(E) \neq 0, 1728$, $j(E) = 1728,$ and $j(E)=0$, respectively (see \cite[Corollary III.10.1]{Sil1}). In this situation, we can obtain the following.

\begin{proposition} \label{ECtwists}
Let $(E, O_E)$ be an elliptic curve over $k$ with $p \geq 5$, 
$N(d)$ the number of $(\overline{k}/k)$-twists of $(E, O_E)$ of degree $d$, $N = \sum_{d \geq 1} N(d)$, and let $M= |\mu_m(k)|$. Then we have $N = M$ and 
$$ N(d) = \left\{ \begin{array}{cc} 
\varphi(d) & \textrm{if}~ d \mid M ,\\
0 & \textrm{otherwise}.
\end{array}\right.$$
\end{proposition}
\begin{proof}
By Hilbert's 90 and the exact sequence of abelian groups
\begin{equation*}
\xymatrix{1 \ar[r]  & \mu_m(k) \ar[r] & k^\times \ar[r]^-{(\cdot)^m} & k^\times  \ar[r] & k^\times/(k^\times)^m \ar[r] & 1,}
\end{equation*}
we know that
$$|H^1(G_k, \mu_m(\overline{k}))| = |k^\times/(k^\times)^m| = |\mu_m(k)|.$$ 
Thus it follows from Proposition \ref{ell rem} that we have
$$ N= | \Twist(\overline{k}/k, (E,O_E) )|=|H^1(G_k, \Aut_{\overline{k}}(E_{\overline{k}}))|=|H^1(G_k, \mu_m(\overline{k}))|=M.$$

By a similar argument as in the above, note that we have the following group isomorphisms
\begin{equation}\label{ell isoms}
\Twist(\overline{k}/k, (E,O_E) ) \cong H^1(G_k, \Aut_{\overline{k}}(E_{\overline{k}})) \cong H^1 (G_k, \mu_m(\overline{k})) \cong H^1(G_k, \mu_M(\overline{k}))=\Hom(G_k, \mu_M(\overline{k})).
\end{equation}

Let $(E^{\prime}, O_{E^{\prime}})$ be a $(L/k)$-twist of $(E,O_E)$ of degree $[L:k]$, as in Definition \ref{degdef},
that corresponds to a  character $\chi \in \Hom(G_k, \mu_M(\overline{k}))$ of order $d$, for some $d~|~M$, via the identifications of (\ref{ell isoms}). Since $H^1 (G_k, \mu_M (\overline{k})) \cong k^{\times}/(k^{\times})^M,$ we may write $\chi = \chi_\alpha$ for some $\overline{\alpha} \in k^\times/(k^\times)^M$, where $\chi_{\alpha}$ is defined by $\chi_\alpha(\sigma) = \sigma(\beta)/\beta$ for $\sigma \in G_k$ and $\beta$ satisfying $\beta^d = \alpha$.
For a $\tau \in G_{k(\beta)}$, we have $\chi_\alpha(\tau) = 1$.
Thus the twist $(E',O_{E'})$ is in the kernel of the restriction map to  $G_{k(\beta)}$,  which means that
it is isomorphic to $(E, O_E)$ over $k(\beta)$, and hence, $[L:k] \leq [k(\beta):k]$. 
Also, since $(E^{\prime},O_{E^{\prime}})$ becomes isomorphic to $(E,O_E)$ over $L$, we have $\sigma(\beta)=\beta$ for all $\sigma \in \textrm{Gal}(\overline{k}/L)$ so that $\beta \in L$, whence, $k(\beta) \subseteq L.$ 
Consequently, it follows that $L=k(\beta)$.
Now, since $\mu_d (\overline{k}) \subset k^{\times}$ and the order of $\chi$ is exactly $d,$ we can see that $[k(\beta):k]=d$, and thus, we obtain $\textrm{deg}(E^{\prime},O_{E^{\prime}})=d.$
Since the number of elements of order $d$ in $\Hom(G_k, \mu_M(\kbar))$ is equal to $\varphi(d)$, it follows that $\varphi(d) \leq N(d)$, and this inequality holds for any divisor $d$ of $M$. Therefore we get
\begin{equation*}
M=\sum_{d^{\prime}|M} \varphi(d^{\prime}) \leq \sum_{d^{\prime}|M} N(d^{\prime}) \leq N =M,
\end{equation*}
and hence, it follows that $\varphi(d^{\prime})=N(d^{\prime})$ for all $d^{\prime}~|~M$, and $N(d^{\prime})=0$ otherwise. \\
\indent This completes the proof.
\end{proof}

\begin{remark}
If $p=2$ or $3,$ then the conclusion of Proposition \ref{ECtwists} might not hold, in general. For such examples, see \cite[Propositions 2.1 and 3.1]{KAT}.
\end{remark}

Now, we  give a higher dimensional analogue of Proposition \ref{ECtwists} for simple polarized abelian varieties whose automorphism group is cyclic.

\begin{theorem} \label{newmain}
Let $(X, \lambda)$ be a simple polarized abelian variety  over $k$ such that $\Aut_{\kbar}(X_{\kbar}, \lambda_{\kbar})$ is a cyclic group of order $n$, and let $s$ be an integer such that an action of $\Frob_k$ on $\Aut_{\kbar}(X_{\kbar}, \lambda_{\kbar})$ is induced by a multiplication by $s$. 
We also define $N(d)$ to be the number of $(\kbar/k)$-twists of $(X, \lambda)$ of degree $d$, and put $N:=\sum_{d\geq 1}N(d)$.
Then:\\
(a) We have $N \mid n$, and $N=n$ if and only if $G_k$ acts trivially on $\Aut_{\kbar}(X_{\kbar}, \lambda_{\kbar})$. \\
(b) For $d \nmid n,$ we have $N(d)=0.$ \\
(c) For $d \mid n$, we have
\begin{equation*}
N(d) =   \prod_{\substack{p \mid d }} \lbrb{
p^{\min(t_{1,p}, t_{3,p}) + \min(t_{2,p}, t_{3,p}) - \min(t_{1,p} + t_{2,p}, t_{3,p})} - p^{\min(t_{1,p}-1, t_{3,p}) + \min(t_{2,p}, t_{3,p}) - \min(t_{1,p} + t_{2,p} - 1, t_{3,p})}},
\end{equation*}
where $t_{1,p}=\ord_p(s-1), t_{2,p}=\ord_{p}(s^{d-1}+\cdots+1),$ and $t_{3,p}=\ord_{p} n.$
\end{theorem}
\begin{proof}
For simplicity, we write $A$ for the finite $G_k$-module $\Aut_{\kbar}(X_{\kbar}, \lambda_{\kbar})$.
For any cyclic $G_k$-module $M$, we have $H^1(G_k, M) \cong M/(\Frob_k-1)M$.
Then it follows from Theorem \ref{Intromain1} that
\begin{equation*}
 N= | \Twist(\overline{k}/k, (X, \lambda))| = |H^1(G_k, A)|=
\left| \frac{ A}{(s-1) A} \right|.
\end{equation*}
Hence we have $N \mid n$, and $N = n$ if and only if $(s-1)A = 0$. Since $n$ is even and $(s,n) = 1$, the condition $(s-1)A=0$ is equivalent to the fact that $s =1.$
This proves (a).

Now, for a field extension $k'/k$ of degree $d$, we have a commutative diagram
\begin{equation*}
\xymatrix{
H^1(G_k, A) \ar[r]^-{\cong} \ar[d]^-{\mathrm{res}} & A/(\Frob_k-1)A \ar[d]^-{\tau_d} \\
H^1(G_{k'}, A) \ar[r]^-{\cong} & A/(\Frob_k^d - 1)A
}
\end{equation*}
where $\tau_d$ is an induced action of $1 + \Frob_k + \cdots + \Frob_k^{d-1}$  on $A$.
Since $k'/k$ is a Galois extension, the kernel of the restriction map is identified with the set of twists of $(X,\lambda)$ that become trivial after base change to $k'$,
by Theorem \ref{Intromain1}.
Hence we can identify the set of twists of $(X,\lambda)$ of degree dividing $d$ with the kernel of the map $\tau_d$.
If $d \nmid n$, then the kernel of the restriction to $k'$ is equal to the kernel of restriction to the 
field of degree $(d,n)$, and hence, there is no twist of $(X,\lambda)$ of degree exactly $d$ in this case. This proves (b).
 
For the rest of the proof, we may assume that $d \mid n$.  The sequence 
\begin{equation*}
0 \to A[\Frob_k-1] \to A[\Frob_k^d - 1] \to A[\tau_d] \to \frac{A}{(\Frob_k-1)A} \to \frac{A}{(\Frob_k^d-1)A} \to \frac{A}{\tau_dA} \to 0
\end{equation*}
is exact. Hence, we have
\begin{equation*}
|\ker(\tau_d)| = \frac{|A[\Frob_k-1]||A[\tau_d]|}{|A[\Frob_k^d-1]|} = \frac{\gcd(s-1,n) \gcd(s^{d-1} + \cdots + 1, n)}{\gcd(s^d - 1, n)}.
\end{equation*}

For simplicity, we put $c(d, s, n):=\frac{\gcd(s-1,n) \gcd(s^{d-1} + \cdots + 1, n)}{\gcd(s^d - 1, n)}$, and  we let $m_p$ denote the quantity $p^{\ord_p(m)}$ for an integer $m$.
Finally, for a prime $p,$ let $t_{1,p}, t_{2,p}, t_{3,p}$ be three integers that satisfy $(s-1)_p = p^{t_1}, (s^{d-1} + \cdots +1)_p = p^{t_2}$, and $n_p = p^{t_3}$. 
Then we have
\begin{equation*}
c(d, s, n)_p = p^{\min(t_{1,p}, t_{3,p}) + \min(t_{2,p}, t_{3,p}) - \min(t_{1,p} + t_{2,p}, t_{3,p})}.
\end{equation*}
Now, suppose that $s^d = 1 + ap^t$ for some integers $a$ and $t$ such that $(a,p) =1$. Then for a positive integer $m$ relatively prime to $p$, we have
\begin{equation*}
s^{dm} = (1 + ap^t)^m = 1 + (ma + N )p^t
\end{equation*}
with $(ma, p) = 1$ and $p^t \mid N$. Hence it follows that $\ord_p(s^{dm}-1) = \ord_p(s^d-1)$, which, in turn, implies that $t_{2,p}$ depends only on $d_p$.
Especially, we have $c(d,s,n)_p = c(d_p, s, n) = c(d_p, s, n_p)$ and $c( \bullet , s, n)$ is multiplicative.
By the M\"obius inversion, we know that the number of the twists of $(X,\lambda)$ of degree exactly $d$ equals
\begin{equation*}
\sum_{d' \mid d} \mu(d')c(\frac{d}{d'}, s, n).
\end{equation*}
Since $c(1,s,n)$ = 1, it follows that
\begin{align*}
\sum_{d' \mid d} & \mu(d')c(\frac{d}{d'}, s, n) 
= \prod_{p \mid d} 
\lbrb{ \sum_{d' \mid d_p} \mu(d') c(\frac{d_p}{d'}, s, n_p)}
= \prod_{p \mid d} \lbrb{c(d_p, s, n_p) - c(\frac{d_p}{p}, s, n_p) }\\
&= \prod_{\substack{p \mid d }} \lbrb{
p^{\min(t_{1,p}, t_{3,p}) + \min(t_{2,p}, t_{3,p}) - \min(t_{1,p} + t_{2,p}, t_{3,p})} - p^{\min(t_{1,p}-1, t_{3,p}) + \min(t_{2,p}, t_{3,p}) - \min(t_{1,p} + t_{2,p} - 1, t_{3,p})}},
\end{align*}
and this proves (c). \\
\indent This completes the proof.
\end{proof}
We note that if $G_k$ acts trivially on $\Aut_{\kbar}(X_{\kbar}, \lambda_{\kbar})$ so that $s=1$, then it follows from Theorem \ref{newmain} that the number of $(\overline{k}/k)$-twists of $(X,\lambda)$ of degree exactly $d$ is equal to $\varphi(d)$ for $d \mid n$.

\begin{remark}\label{ell prop ded}
Suppose that $\textrm{char}~k =p \geq 5$. In this case, we have seen that the automorphism group of any elliptic curve $(E, O_E)$ is cyclic, and hence, we can also deduce Proposition \ref{ECtwists} from Theorem \ref{newmain}.
\end{remark}

Now, we recall that a simple abelian variety over $k$ is \emph{absolutely simple (or geometrically simple)} if it is simple over $\overline{k}.$ We first record one useful lemma.
\begin{lemma} \label{abs simple lem-1}
Let $(X,\lambda)$ be an absolutely simple polarized abelian variety of odd prime dimension $g$ over a finite field $k$ with $\textrm{char}~k =p >0.$ If $g=3$ or if $g \geq 5$ and $p \geq 5,$ then $\mathrm{Aut}_{\overline{k}}(X_{\overline{k}}, \lambda_{\overline{k}})$ is cyclic. 
\end{lemma}
\begin{proof}
Let $D=\textrm{End}_{\overline{k}}^0(X_{\overline{k}}).$  Then, in view of \cite[$\S7$]{Oort1}, we have the following three cases to consider: \\
\indent (i) If $D$ is a field, then it is clear that $\textrm{Aut}_{\overline{k}}(X_{\overline{k}}, \lambda_{\overline{k}})$ is cyclic. \\
\indent (ii) If $g \geq 5$ and $D=D_{p,\infty}$ for some prime $p \geq 5$, then it also follows that $\textrm{Aut}_{\overline{k}} (X_{\overline{k}}, \lambda_{\overline{k}})$ is cyclic. Finally, \\
\indent (iii) If $D$ is a central simple division algebra of degree $g$ over an imaginary quadratic field, then it follows from \cite[Theorem 3.5]{Hwa2} that $\textrm{Aut}_{\overline{k}}(X_{\overline{k}}, \lambda_{\overline{k}})$ is cyclic. \\
\indent This completes the proof.
\end{proof}

\begin{remark}
In the case of $p=2$ or $p=3,$ if we add an additional condition that $|\mathrm{Aut}_{\overline{k}}(X_{\overline{k}}, \lambda_{\overline{k}})| \leq 6,$ then we can obtain a similar result as in Lemma \ref{abs simple lem-1}.
\end{remark}

Now, we are ready to prove Theorem \ref{main theorem}.

\begin{proof}[Proof of Theorem \ref{main theorem}]
This follows from Lemma \ref{abs simple lem-1} and Theorem \ref{newmain}.
\end{proof}

We conclude this section by giving a concrete example, which is based on \cite[Theorem 4.1]{Hwa2}.
\begin{example}
Let $k=\mathbb{F}_{125}$. Also, let $X$ be an abelian threefold over $k$ such that $\pi_X$ is conjugate (as $125$-Weil numbers) to a zero of the quadratic polynomial $t^2 +5t+125 \in \mathbb{Z}[t]$, and that $\End_k(X)$ is a maximal $\mathbb{Z}$-order, containing $\mathbb{Z} \left[ \frac{1+\sqrt{-19}}{2}\right]$, in a central simple division algebra $\End_k^0 (X)$ of degree $3$ over $\mathbb{Q}(\sqrt{-19})$, and let $(X^{\prime}, \lambda )$ be a polarized abelian threefold over $k$ with the property that $X^{\prime}$ is $k$-isogenous to $X$ and $\Aut_k(X^{\prime}, \lambda ) \cong \mathbb{Z}/2\mathbb{Z}$ (as in the proof of \cite[Theorem 4.1]{Hwa2}-(1)).  By \cite[Proposition 3]{How2}, $X^{\prime}$ is absolutely simple. Now, it is clear that $\End_k(X^{\prime}) \subseteq \End_{\overline{k}}(X^{\prime}_{\overline{k}}).$ Then since $\End_{\overline{k}}(X^{\prime}_{\overline{k}})$ is a $\mathbb{Z}$-order in $\End_{\overline{k}}^0(X^{\prime}_{\overline{k}})= \End_k^0(X^{\prime})$, and $\End_k(X^{\prime})$ is a maximal $\mathbb{Z}$-order in $\End_k^0(X^{\prime}),$ it follows that $\End_k(X^{\prime})=\End_{\overline{k}}(X^{\prime}_{\overline{k}}).$ In particular, we have $\Aut_k (X^{\prime})=\Aut_{\overline{k}}(X^{\prime}_{\overline{k}}).$ Therefore we can conclude that 
$$\Aut_{\overline{k}}(X^{\prime}_{\overline{k}}, \lambda_{\overline{k}})=\Aut_k (X^{\prime},\lambda) \cong \mathbb{Z} / 2\mathbb{Z}.$$ Then by Theorem \ref{main theorem}, we have $|\Twist(\overline{k}/k, (X^{\prime},\lambda))|=2$ and the only non-trivial $(\overline{k}/k)$-twist of $(X^{\prime},\lambda)$ is of degree $2.$
\end{example}

\textbf{Acknowledgment.} The first author was supported by a KIAS Individual Grant (MG069901) at Korea Institute for Advanced Study. The second author was supported by the National Research Foundation of Korea (NRF) funded by the Ministry of Education, under the Basic Science Research Program(2019R1C1C1004264). He thanks KIAS for its hospitality and financial support, and Junyeong Park for useful discussion. Also, the authors give thanks to Jungin Lee for letting them know about a reference.

\end{document}